\documentclass[10pt]{article}
\font\smallit=cmti10

\usepackage{amssymb,latexsym,amsmath,epsfig,amsthm} 
\usepackage[utf8]{inputenc}
\usepackage{graphicx}
\usepackage{mathrsfs}
\usepackage{stmaryrd}

\usepackage{amsthm}  

\usepackage{subfig}

\makeatletter

\renewcommand\section{\@startsection {section}{1}{\z@}
{-30pt \@plus -1ex \@minus -.2ex}
{2.3ex \@plus.2ex}
{\normalfont\normalsize\bfseries}}

\renewcommand\subsection{\@startsection{subsection}{2}{\z@}
{-3.25ex\@plus -1ex \@minus -.2ex}
{1.5ex \@plus .2ex}
{\normalfont\normalsize\bfseries}}

\renewcommand{\@seccntformat}[1]{\csname the#1\endcsname. }

\makeatother

\newtheorem{theorem}{Theorem}[subsection]
\newtheorem*{theo}{Theorem}
\newtheorem{lemma}[theorem]{Lemma}
\newtheorem{proposition}[theorem]{Proposition}

\theoremstyle{definition}
\newtheorem{definition}[theorem]{Definition}

\newtheorem{exmp}[theorem]{Example}

\newtheorem{remark}[theorem]{Remark}

\newtheorem{example}[theorem]{Example}

\usepackage{color}
\usepackage{graphicx}

\def\vv#1#2{\begin{pmatrix} #1 \\[5pt] #2 \end{pmatrix}}
\def\mm#1#2#3#4{\begin{pmatrix} #1 & #2 \\[5pt] #3 & #4 \end{pmatrix}}

\def\hA{\hat A}

\begin{document}

\begin{center}
\uppercase{\bf On the asymptotic behavior of density of sets defined by sum-of-digits function in base 2}
\vskip 20pt
{\bf Jordan Emme }\\
{\smallit Aix-Marseille Université, CNRS, Centrale Marseille, I2M, UMR 7373, 13453 Marseille, France.}\\
{\tt jordan.emme@univ-amu.fr}\\ 
\vskip 10pt
{\bf Alexander Prikhod'ko }\\
{\smallit Moscow Institute of Physics and Technology, Moscow, Russia.}\\
{\tt sasha.prihodko@gmail.com}\\ 
\end{center}
\vskip 30pt
\vskip 30pt

\centerline{\bf Abstract}
\noindent
Let $s_2(x)$ denote the number of occurrences of the digit ``$1$'' in the binary expansion of  $x$ in $\mathbb{N}$. 
We study the mean distribution $\mu_a$ of the quantity $s_2(x+a)-s_2(x)$ 
for a fixed positive integer $a$.
It is shown that solutions of the equation
$$ s_2(x+a)-s_2(x)= d $$
are uniquely identified by a finite set of prefixes in $\{0,1\}^*$, 
and that the probability distribution of differences $d$ is given by 
an infinite product of matrices whose coefficients are operators of $l^1(\mathbb{Z})$.
Then, denoting by $l(a)$ the number of occurrences of the pattern ``$01$'' in the binary expansion of $a$, we give the asymptotic behavior of this probability distribution as $l(a)$ goes to infinity, as well as estimates of the variance of the probability measure $\mu_a$.


\section{Introduction}

\subsection{Background}

In this article we are interested in the statistical behavior of the difference of the number of digits $1$ in the binary expansion of an integer $x$ before and after its summation with $a$. This kind of question can be linked with carry propagation problems developed in \cite{knuth} and \cite{pippenger}, and is linked to computer arithmetic as in \cite{muller} or \cite{ercegovac}, but our approach is different.

The main definitions of the paper are the following.

\begin{definition}
 For every integer $x$ in $\mathbb{N}$ whose binary expansion is given by
 $$x=\sum_{k=0}^{n} x_k 2^k,\quad x_k \in \{0,1\},$$
 we define the quantity 
 $$s_2(x)=\sum_{k=0}^{n} x_k,$$
 which is the number of occurrences of the digit $1$ in the binary expansion of $x$.
\end{definition}

\begin{definition}
 For every integer $x$ in $\mathbb{N}$ whose binary expansion is given by:
 $$x=\sum_{k=0}^{n} x_k 2^k,$$
 we denote by $\underline{x}$ the word $x_0...x_n$ in $\{0,1\}^{*}$.
\end{definition}

\begin{remark}
 Since we are working in the free monoid $\{0,1\}^{*}$ of binary words, let us state right away that we will denote the cylinder set of a word $w=w_0...w_n$ by the standard notation $[w]:= \{ v \in \{0,1\}^{\mathbb{N}} \quad | \quad v_0...v_n=w_0...w_n \}$. These sets form a topological basis of clopen sets of  $\{0,1\}^{\mathbb{N}}$ for the product topology. Moreover, we endow the set of binary configurations $\{0,1\}^{\mathbb{N}}$ with the natural probability measure $\mathbb{P}$ that is the balanced Bernoulli probability measure defined on the Borel sets.
\end{remark}

The function $s_2$ modulo 2 was extensively studied for its links with the Thue-Morse sequence (as found in \cite{keane}), for example, or for arithmetic reasons as in \cite{MauduitRivat1} and \cite{MauduitRivat2}. In this paper our motivation is not of the same nature. Thus we will not look at $s_2$ modulo 2, but instead at the function $s_2$, which is the sum of the digits in base 2.

In \cite{Besineau}, Bésineau studied the statistical independence of sets defined via functions such as $s_2$, \emph{i.e.} ``sum of digits'' functions. To this end he studied the correlation function defined in the following way.
\begin{definition}
 Let $f:\mathbb{N}\rightarrow \mathbb{C}$. Its correlation $\gamma_f:\mathbb{N}\rightarrow\mathbb{C}$ is the function, if it exists, defined by
 $$
 \gamma_f(a)=\lim_{N\rightarrow \infty} \frac{1}{N}\sum_{k=0}^N\overline{f(k)}f(k+a).
 $$
\end{definition}
This correlation was studied in \cite{Besineau} for functions of the form $k\mapsto e^{i\pi \alpha s(k)}$ where $s$ is a sum-of-digits function in a given base.

In the case of base 2, this motivates us to understand the following equation, with parameters $a$ in $\mathbb{N}$ and $d$ in $\mathbb{Z}$:
\begin{equation}\label{mainequation}
 s_2(x+a)-s_2(x)=d.
\end{equation}
Namely, we wish to understand, for given integers $a$ and $d$, the behavior of the density of the set $\{x \in \mathbb{N} \ | \ s_2(x+a)-s_2(x)=d \},$ which always exists (from \cite{Besineau, MorgenbesserSpiegelhofer} for instance).
We remark that we can define in this way, for every integer $a$, a probability measure $\mu_a$ on $\mathbb{Z}$  defined, for every $d$, by
$$
\mu_a(d):=\lim_{N\rightarrow \infty} \frac1N \quad \#  \left\{ x \leq N \quad | \quad s_2(x+a)-s_2(x)=d\right\}.
$$

\begin{exmp}
 Let us compute $\mu_1$ and deduce $\mu_{2^n}$ for every $n$.
 
 First, we remark that  $s_2(x+1)-s_2(x)\leq 1$ for every non-negative integer $x$, since there is only one ``1'' in the binary expansion of 1.
 
 It is easy to see that $\mu_1(1)$ is the density of the set of even integers, so we have \mbox{$\mu_1(1)=\frac12$}. We also remark that in order to get a difference $d \leq 0$, one has to propagate a carry in the summation $d+1$ times. In other words, suppose we wish to have  \mbox{$s_2(x+1)-s_2(x)=d$} with a negative $d$; then $\underline{x}$ has to start with $1^{-d+1}0$. This means that $x\equiv 2^{-d+1}-1 \ \text{mod}\ (2^{-d+2})$. The set $\{x \in \mathbb{N}\ |\ x\equiv 2^{-d+1}-1 \ \text{mod}\ (2^{-d+2})\}$ is of density $\frac{1}{2^{-d+2}}$.
 
 Finally, 
 $$
 \mu_1(d)=\left\{ \begin{array}{ccc}
                  0 & \text{if} & d>1\\
                  \frac{1}{2^{-d+2}}& \text{if} & d\leq1.
                 \end{array}\right.
 $$
 
 Now notice  that $\mu_{2a}=\mu_a$ for every $a$. This is proved later with Proposition~\ref{p.proba_rec}. Hence, for every $n$ in $\mathbb{N}$,
  $$
 \mu_{2^n}(d)=\left\{ \begin{array}{ccc}
                  0 & \text{if} & d>1\\
                  \frac{1}{2^{-d+2}}& \text{if} & d\leq1.
                 \end{array}\right.
 $$
\end{exmp}

Section~2 studies the set $\{x \in \mathbb{N} \ | \ s_2(x+a)-s_2(x)=d \}$ with a combinatorial approach, using constructions from theoretical computer science such as languages, graphs, and automata. In this way we prove that the probability $\mu_a$ is given by a product of matrices.

Section~3 is a technical part that allows us to understand the asymptotic behavior of such a measure as $a$ gets bigger in a certain, non-trivial, sense. This section develops the idea that our probability measure gets smaller as $l(a)$ increases.
Namely, we give an upper bound of the $l^2(\mathbb{Z})$ norm of the probability measure $\mu_a$ depending on the number of occurrences of ``$01$'' in $\underline{a}$. 

Finally, we give bounds on the variance of $\mu_a$, also depending on the number of patterns ``01'' in $\underline{a}$.

%

\subsection{Results}

The main results are the following.
\begin{theo}
 The distribution $\mu_a$ is calculated via 
an infinite product of matrices whose coefficients are operators of $l^1(\mathbb{Z})$ applied to a vector whose coefficients are elements of $l^1(\mathbb{Z})$:
$$ \mu_a = ( Id,\; Id ) \;  
    \cdots A_{a_n} A_{a_{n-1}} \cdots A_{a_1} A_{a_0} 
    \begin{pmatrix} \delta_0 \\ 0 \end{pmatrix},
$$
where the sequence $(a_n)_{n\in \mathbb{N}}$ is the binary expansion of~$a$,  $\delta_0$ is the Dirac mass in $0$, the $A_i$ are defined by
$$
    A_0 = \begin{pmatrix} Id & \frac12 {S^{-1}} \\ 0 & \frac12 {S} \end{pmatrix}, \qquad 
    A_1 = \begin{pmatrix} \frac12 {S^{-1}} & 0 \\ \frac12 {S} & Id \end{pmatrix}, 
$$
and $S$ is the left shift transformation on~$l^1(\mathbb{Z})$. 
\end{theo}

\begin{remark}
 We recall that the set of finite measures on $\mathbb{Z}$ is in bijection with the elements of $l^1(\mathbb{Z})$. We will always identify finite measures on $\mathbb{Z}$ and elements of $l^1(\mathbb{Z})$.
\end{remark}

Such a result allows an analytical study of these distributions as the binary expansion of $a$ becomes more and more `complicated.' Let us formalize this notion of complexity for $a$.

\begin{definition}
 For every $a$ in $\mathbb{N}$, let us denote by $l(a)$ the number of subwords $01$ in the binary expansion of $a$.
\end{definition}

\begin{remark}
 We advise the reader to be careful, as this notion of `complexity' has nothing to do whatsoever with the classical one in the field of language theory.
\end{remark}

We are thus interested in what happens as there are more and more patterns $01$ in the word $\underline{a}$. We can estimate precisely the asymptotic behavior of the $l^2(\mathbb{Z})$ norm  $\|\cdot\|_2$ of this distribution as $a$ goes to infinity by increasing the number of subwords $01$.

Namely, we have the following theorem.
\begin{theo}
There exists a real constant $C_0$ such that for every integer $a$ we have the following:
$$
  \|\mu_a\|_2 \le C_0 \cdot l(a)^{-1/4}.
$$
\end{theo}
Finally, in the last section, we wish to obtain a much more precise result regarding the behavior of such a distribution than just estimates of the $l^2$ norm. So we study the way in which the variance of the random variable of probability law $\mu_a$ is linked to the number of subwords that are $01$ in the binary expansion of $a$.
We define the variance of $\mu_a$ to be the quantity
$$
\text{Var}(a)=\sum_{d \in \mathbb{Z}} \mu_a(d)(n-m_{a,1})^2,
$$
where $m_{a,1}$ is the first moment (the mean) of the probability measure $\mu_a$.
It is shown in this last section that these moments (the mean and the variance) do exist; in fact, the mean is always zero. Moreover, we have bounds on this variance as shown in the following result.
\begin{theo}
 For every integer $a$ such that $l(a)$ is large enough, the variance of $\mu_a$, denoted by $\text{Var}(a)$,  has bounds
 
 $$
 l(a)-1 \leq \text{Var}(a) \leq 2(2 l(a)+1).
 $$
\end{theo}

\section{Statistics of binary sequences}

In this section we wish to understand the following quantity: 
$$
\mu_a(d):=\lim_{N\rightarrow \infty} \frac1N \quad \#  \left\{ x \leq N \quad | \quad s_2(x+a)-s_2(x)=d\right\},
$$
for every given positive integer $a$ and every integer $d$.

We know from \cite{Besineau} that such a limit exists, and it has been studied in \cite{MorgenbesserSpiegelhofer}, for example, but here we give a proof using the structure of solutions of the following equation:
$$
s_2(x+a)-s_2(x)=d
$$
for every $a$ and $d$.

Let us investigate such solutions, as well as their construction, in order to understand the distribution of probability $\mu_a$  of differences $d= s_2(x+a)-s_2(x)$. We prove that this distribution is given by an infinite product of matrices whose sequence is given by the binary expansion of $a$. 

\subsection{Combinatorial description of summation tree}

First we prove the following lemma.

\begin{lemma}\label{lemma.prefixes}
  \it For all $a$ in $\mathbb{N}$ and $d$ in $\mathbb{Z}$, there exists a finite set of words \mbox{$\mathcal{P}_{a,d}=\{p_1,...,p_k\}\subset \{0,1\}^*$} such that $x$ is solution of (\ref{mainequation}) if and only if
  $$\underline{x} \in \bigcup_{j=1}^k [p_j]$$
  where, for every word $w$ in $\{0,1\}^*$, $[w]$ denotes the cylinder set of words whose prefix is $w$.
\end{lemma}

\medskip

\begin{remark}
 Before we prove Lemma~\ref{lemma.prefixes}, let us remark that for every even integer $n$ the following holds:
 $$
 s_2(n)=s_2\left(\frac{n}{2}\right).
 $$
 We also remark that, for every odd integer $n$, the following holds:
 $$
  s_2(n)=s_2\left(\frac{n-1}{2}\right)+1.
 $$
\end{remark}

\begin{proof}
 Let us prove this lemma by induction on $a$. It is obviously true that for $a=1$, there exists a (possibly empty) set of words for every $d$ in $\mathbb{Z}$ describing the solutions of Equation~(\ref{mainequation}). Let us assume that it is true for every integer not greater than some given $a$ in $\mathbb{N}$. Now let $k$ be an integer not greater than $2a$.
 \begin{itemize}
  \item  If $k$ is even then, for every $d$, $\mathcal{P}_{k,d}=\{0w, w\in \mathcal{P}_{\frac{k}{2},d}\}\cup\{1w, w\in \mathcal{P}_{\frac{k}{2},d}\}$. 

 Indeed, $x$ being an even solution of $s_2(x+k)-s_2(x)=d$ is equivalent to having $$s_2\left(\frac{x+k}{2}\right)-s_2\left(\frac{x}2\right)=d$$ which can be written as  $$s_2\left(\frac{x}2+\frac{k}2\right)-s_2\left(\frac{x}2\right)=d.$$  From the induction hypothesis it follows that $\underline{x}$ starts with a $0$ followed by a word in $\mathcal{P}_{\frac{k-1}{2},d}$.
 
 Moreover, $x$ being an odd solution of $s_2\left(x+k\right)-s_2\left(x\right)=d$ is equivalent to having $$s_2\left(\frac{x-1+k}2\right)+1-\left(s_2\left(\frac{x-1}2\right)+1\right)=d$$ which can be written as $$s_2\left(\frac{x-1}2+\frac{k}2\right)-s_2\left(\frac{x-1}2\right)=d.$$ Then, by the induction hypothesis, $\underline{x}$ must start with a $1$ followed by a word in $\mathcal{P}_{\frac{k}{2},d}$.
 
 \item If, however, $k$ is odd, then  $\mathcal{P}_{k,d}=\{0w, w\in \mathcal{P}_{\frac{k-1}{2},d-1}\}\cup\{1w, w\in \mathcal{P}_{\frac{k+1}{2},d+1}\}.$
 
 Indeed, $x$ being an even solution of $s_2\left(x+k\right)-s_2\left(x\right)=d$ is equivalent to having $$s_2\left(\frac{x+k-1}{2}\right)+1-s_2\left(\frac{x}2\right)=d$$ which can be written as $$s_2\left(\frac{x}2+\frac{k-1}2\right)-s_2\left(\frac{x}2\right)=d-1.$$  From the induction hypothesis it follows that $\underline{x}$ starts with a $0$ followed by a word in $\mathcal{P}_{\frac{k-1}{2},d-1}$.
 
  Moreover, $x$  being an odd solution of $s_2\left(x+k\right)-s_2\left(x\right)=d$ is equivalent to having $$s_2\left(\frac{x-1+k+1}2\right)-\left(s_2\left(\frac{x-1}2\right)+1\right)=d$$ which can be written as $$s_2\left(\frac{x-1}2+\frac{k+1}2\right)-s_2\left(\frac{x-1}2\right)=d+1.$$ Then, by the induction hypothesis, we can state that $\underline{x}$ must start with a $1$ followed by a word in $\mathcal{P}_{\frac{k+1}{2},d+1}$.
 \end{itemize}

Hence knowing the sets $\mathcal{P}_{1,d}$ for every integer $d$ allows to inductively compute the sets $\mathcal{P}_{a,d}$ for every $a$ in $\mathbb{N}$.
\end{proof}

\begin{remark}
 It should be noted that such a set is not unique. For instance, the sets $\{01,00\}$ and $\{0\}$ can both qualify as $\mathcal{P}_{0,1}$ with the lemma's notations.
\end{remark}

From the previous lemma it follows immediately that, for all integers $d$ and positive integers $a$, the sequence $\left(\#  \left\{ x \leq N  |  s_2(x+a)-s_2(x)=d\right\}\right)_{N\in\mathbb{N}}$ can be written as a sum of sequences, all of which contain an arithmetic progression. Thus the following limit exists:
$$
\mu_a(d):=\lim_{N\rightarrow \infty} \frac1N \quad \#  \left\{ x \leq N \quad | \quad s_2(x+a)-s_2(x)=d\right\}.
$$
Moreover, a quick computation yields
$$
\mu_a(d)=\mathbb{P} \left(\bigcup_{p \in \mathcal{P}_{a,d}}[p]\right)
$$
where $\mathbb{P}$ is the balanced Bernoulli probability measure on $\{0,1\}^{\mathbb{N}}$.

\begin{remark}
 In all that follows, we always consider sets $\mathcal{P}_{a,d}$ such that for every words $ p_1,p_2$ in $\mathcal{P}_{a,d}$,
\begin{equation}\label{cond.prefixes}
 p_1\neq p_2 \implies [p_1]\cap[p_2]=\emptyset.
\end{equation}
So in this case,
$$
\mu_a(d)=\sum_{p \in \mathcal{P}_{a,d}}\mathbb{P}([p]).
$$
\end{remark}

\begin{remark}\label{r.prefixes}
 Note that a sufficient condition for a set $\mathcal{P}_{a,d}$ to satisfy (\ref{cond.prefixes}) is that all the prefixes be of same length since two cylinders defined by prefixes of the same length are disjoint whenever the prefixes are different.
\end{remark}

The proof of Lemma~\ref{lemma.prefixes} naturally gives the idea of an inductive way to compute the prefixes. A comfortable way to do that is to span, for every $a$, a tree $\tau_a$ in such a way that the family of trees obtained is consistent with the following recursive properties for every $k$:
$$
\mathcal{P}_{2k,d}=\{0w, w\in \mathcal{P}_{k,d}\}\cup\{1w, w\in \mathcal{P}_{k,d}\}
$$
and
$$
\mathcal{P}_{2k+1,d}=\{0w, w\in \mathcal{P}_{k,d-1}\}\cup\{1w, w\in \mathcal{P}_{k+1,d+1}\}
$$
\begin{lemma}\label{l.tree}
 For each $a$ in $\mathbb{N}$, there exists a tree $\tau_a$ with vertices labeled in $\{0,...,a\}\times\mathbb{Z}$ and edges labeled in $\{0,1\}$ such that, for every $d$ in $\mathbb{Z}$, words in $\mathcal{P}_{a,d}$ are exactly paths labels from the vertex $(a,0)$ to a vertex $(0,d)$.
\end{lemma}

\begin{proof}

Let us prove this by induction on $a$. First, assume that $a=1$. The tree $\tau_1$ is given by Figure~\ref{f.subtree1}. Vertices labels are a pair, the first coordinates are boxed (in the vertex), the second coordinates of the labels are circled. The edges labels (0 or 1) are indicated above the edges. One can check that the path from the vertex labeled $(1,0)$ (which is the root in this case) to a vertex $(0,d)$ defines a sequence of edges whose labels form a word which is the only word in $\mathcal{P}_{1,d}$. For example, if $s_2(x+1)-s_2(x)=-2$, then $\underline{x}$ indeed begins with the word ``1110''.

\begin{figure}[ht!]\centering
 \includegraphics[scale=0.35]{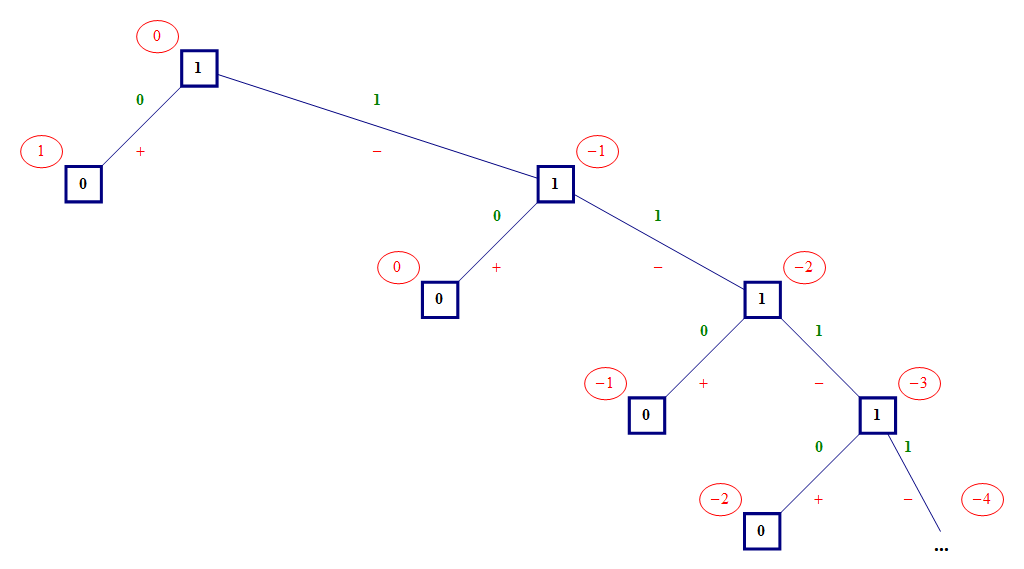}
 \caption{Infinite subtree spanned by every vertex labeled by $1$}
 \label{f.subtree1}
\end{figure}

Let us now assume that every such tree exists until a fixed integer $a$. Let $k$ be an integer no greater than $2a$.

\begin{itemize}
 \item If $k$ is even, we add to the tree $\tau_{k/2}$ a vertex labeled by $(k,0)$ and add two edges between the vertices $(k,0)$ and $(\frac{k}{2},0)$: one labeled by a $0$ and the other one by a $1$. This tree does satisfy the property that the set $\mathcal{P}_{k,d}$ is the set of paths labels from $(k,0)$ to a vertex $(0,d)$. Indeed, 
 $$
\mathcal{P}_{k,d}=\{0w, w\in \mathcal{P}_{k/2,d}\}\cup\{1w, w\in \mathcal{P}_{k/2,d}\},
$$
and starting from vertex $(k,0)$, one can choose either edge (0 or 1) to get to the vertex $(\frac{k}{2},0)$, from which the induction hypothesis holds.
 \item If $k$ is odd, we define $\widetilde{\tau_{\frac{k-1}2}}$ to be a copy of the tree $\tau_{\frac{k-1}2}$ where every vertex labeled $(n,c)$ is relabeled $(n,c+1)$. We also denote by $\widetilde{\tau_{\frac{k+1}2}}$ the copy of $\tau_{\frac{k+1}2}$ where every vertex is relabeled from $(n,c)$ to $(n,c-1)$. The tree $\tau_k$ is obtained by taking the union of  $\widetilde{\tau_{\frac{k-1}2}}$ and $\widetilde{\tau_{\frac{k+1}2}}$, adding a vertex labeled $(k,0)$ and two edges: one labeled $0$ between $(k,0)$ and $(\frac{k-1}2,1)$ and one labeled 1 between $(k,0)$ and $(\frac{k+1}2,-1)$. We recall that 
 $$
 \mathcal{P}_{k,d}=\{0w, w\in \mathcal{P}_{\frac{k-1}{2},d-1}\}\cup\{1w, w\in \mathcal{P}_{\frac{k+1}{2},d+1}\}.
 $$
 Hence using the induction hypothesis on both $\tau_{\frac{k+1}2}$ and $\tau_{\frac{k-1}2}$ yields the result for $\tau_k$.
\end{itemize}

Thus for every integer $k$ there exists a tree $\tau_k$ which satisfies the property described in the lemma.
\end{proof}

If one wishes to construct such a tree for a given positive integer $a$, let us give a method to easily do so (which is a direct consequence of the proof).

Let us choose a particular $a$ in $\mathbb{N}$ and construct the associated tree $\tau_a$. 
The construction starts at level $0$ with only one vertex labeled by $(a,0)$ and we define the tree level by level. If $\tau_a$ is defined up to level $n$ in $\mathbb{N}$, from every vertex we construct the vertices at level $n+1$ in the following way.

Starting from a vertex labeled by $(k,c)$:
\begin{itemize}
 \item if $k$ is even we add a vertex at level $n+1$ labeled by $(\frac{k}{2},c)$ and add one edge of each type between these vertices;
 \item if, however, $k$ is odd, we add two vertices labeled by $(\frac{k-1}{2},c+1)$ and $(\frac{k+1}{2},c-1)$ and we construct an edge of type $0$ between $(k,c)$ and $(\frac{k-1}{2},c+1)$ and an edge of type $1$ between $(k,c)$ and $(\frac{k+1}{2},c-1)$. 
\end{itemize}
 
Notice that the second coordinate of the labels of the vertices can be seen as a counter that keeps track of the difference $s_2(x+a)-s_2(x)$ as the summation is processed digit by digit. More precisely, if we actually compute by hand the sum of $x$ and $a$, the value of the counter on the $n^{th}$ level is exactly the number of $1$'s lost or gained when the $n^{th}$ digit of the sum is computed. Once we reach a vertex $(0,d)$, the summation is completed.

We can see an example of such a construction on Figure~\ref{f.tree5} for $a=5$.

\begin{example}
\begin{figure}[ht!]\centering
 \includegraphics[scale=0.25]{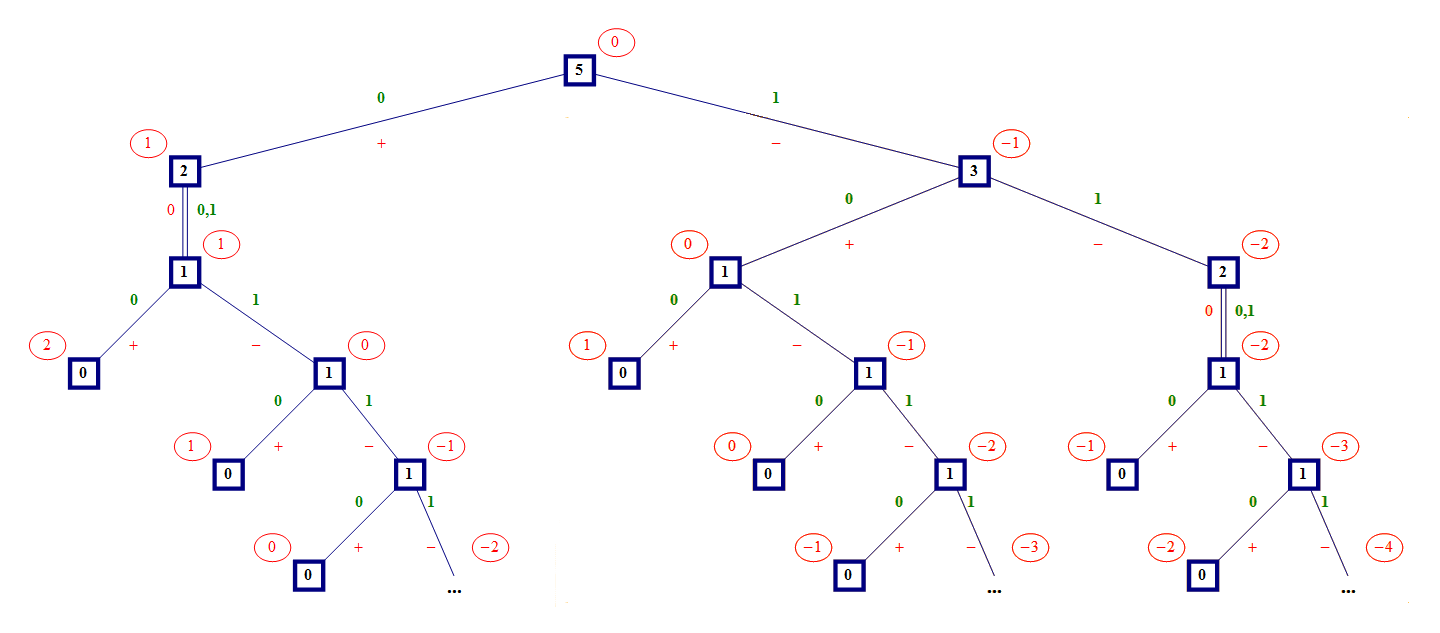}
 \caption{Part of the tree $\tau_5$.}
 \label{f.tree5}
\end{figure}
 Let $a=5$. Then we start our construction on the first level of the tree by creating the root labeled by $(5,0)$. The number $5$ being odd, we add the vertices (2,1) and (3,-1) on the second level, and add an edge of type 0 between vertices (2,1) and (5,0) and an edge of type 1 between vertices (3,-1) and (5,0).
 
 We continue the construction on the third level by:
 \begin{itemize}
  \item adding a child to vertex (2,1) labeled by (1,1), since 2 is even;
  \item adding two children to vertex (3,-1) labeled by vertices (1,0) and (2,-2), since 3 is odd. The edge between vertices (1,0) and (3,-1) is of type 0 and the edge between vertices (2,-2) and (3,-1) is of type 1.
 \end{itemize}

 We continue this process on each vertex at each level to obtain the tree of Figure~\ref{f.tree5}.

%

Let us compute the words in $\mathcal{P}_{5,0}$. It can be read from the tree $\tau_{5}$ that this set is given by the words 00110, 01110, and 1010.

We wish to remark that, because of the way $\underline{x}$ is defined, the word 00110, for example, is not the binary expansion of 6 but of 12, as everything is mirrored.
\end{example}

We now remark that having such a family of trees proves the following proposition.

\begin{proposition}\label{p.proba_rec}
 \it For every $a$ in $\mathbb{N}$ and every $d$ in $\mathbb{Z}$, we have the following identities:
 $$
 \mu_{2a}(d)=\mu_a(d)
 $$
and
 $$
 \mu_{2a+1}(d)=\frac{1}{2}\mu_a(d-1)+\frac{1}{2}\mu_{a+1}(d+1).
 $$
\end{proposition}

\begin{proof}
 For every positive integer $a$ and every integer $d$, we can read from the tree $\tau_{2a}$ that words in $\mathcal{P}_{2a,d}$ are exactly words beginning with either $0$ or $1$, then followed by a word in $\mathcal{P}_{a,d}$. Thus we have
  $$
 \mu_{2a}(d)=\mu_a(d).
 $$
 Notice that words in $\mathcal{P}_{2a+1,d}$ are exactly the words beginning with a 0 followed by a word in $\mathcal{P}_{a,d-1}$ and the words beginning with a 1 followed by a word in $\mathcal{P}_{a+1,d+1}$. It follows that
  $$
 \mu_{2a+1}(d)=\frac{1}{2}\mu_a(d-1)+\frac{1}{2}\mu_{a+1}(d+1).
 $$
\end{proof}

\begin{remark}
 This proposition allows us to explicitly compute every distribution $\mu_a$, since it is trivial to give explicit closed formulas for $\mu_0$ and $\mu_1$. It also gives an understanding of the link between $\mu_a$ and the binary expansion of $a$. However, we prefer another presentation of such a result. Namely, one that assembles in a close formula these identities as we follow the binary decomposition of $a$. Such a formula is given in Theorem~\ref{thm.distrib}.
\end{remark}

One way to do such a thing is to `collapse' the tree $\tau_a$ in order to obtain a `collapsed' graph. This is more or less done by identifying the vertices whose label's first coordinates are the same, and we obtain a graph whose structure is studied in the next section.

%
%

Let us define properly this `collapsed' graph for an integer $a$.
\begin{itemize}
 \item The graph has levels indexed on $\mathbb{N}$.
 \item For every $n$ in $\mathbb{N}$, the level $n$ has exactly two vertices labeled by $\lfloor \frac{a}{2^n}\rfloor$ and $\lfloor \frac{a}{2^n}\rfloor+1$.
 \item There are edges only between vertices on consecutive levels and they are described in the following way:
 \begin{itemize}
 \item if $\lfloor\frac{a}{2^n}\rfloor$ is even then there are two edges between $\lfloor\frac{a}{2^n}\rfloor$ and $\lfloor\frac{a}{2^{n+1}}\rfloor$. There is also an edge between $\lfloor \frac{a}{2^n}\rfloor+1$ and $\lfloor \frac{a}{2^{n+1}}\rfloor$, and another between $\lfloor \frac{a}{2^n}\rfloor+1$ and $\lfloor \frac{a}{2^{n+1}}\rfloor+1$;
  \item if $\lfloor\frac{a}{2^n}\rfloor$ is odd then there are two edges between $\lfloor\frac{a}{2^n}\rfloor+1$ and $\lfloor\frac{a}{2^{n+1}}\rfloor+1$. There is also an edge between $\lfloor \frac{a}{2^n}\rfloor$ and $\lfloor \frac{a}{2^{n+1}}\rfloor$, and another between $\lfloor \frac{a}{2^n}\rfloor$ and $\lfloor \frac{a}{2^{n+1}}\rfloor+1$;
 \end{itemize}
\end{itemize}

\begin{figure}[ht!]
\centering
 \includegraphics[scale=0.33]{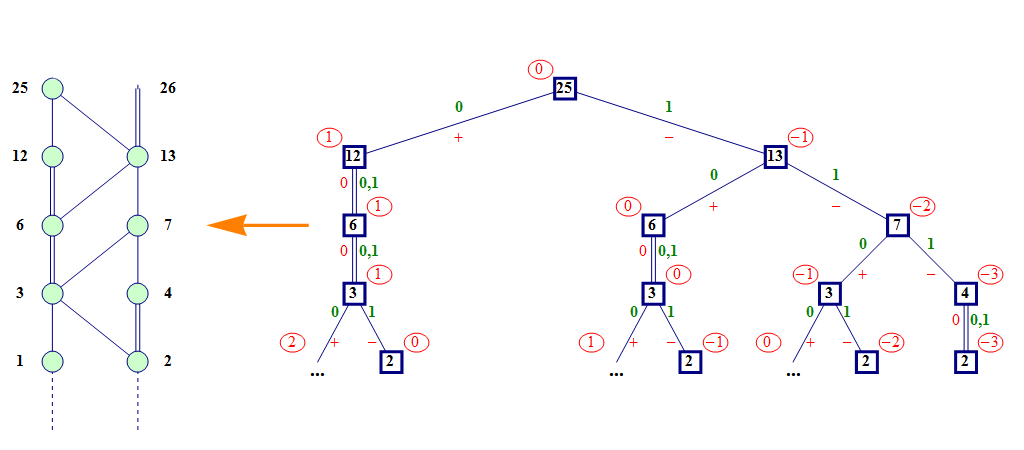}
 \caption{Collapsing $\tau_{25}$.}
 \label{f.collapse}
\end{figure}

\begin{figure}[ht!]
\centering
 \includegraphics[scale=0.31]{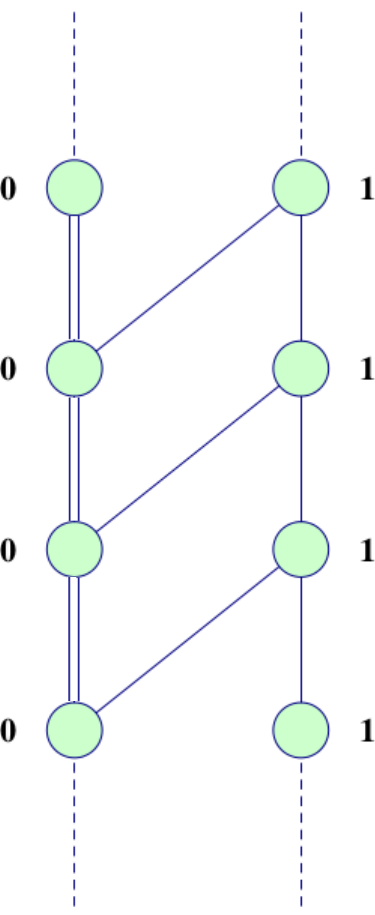}
 \caption{``Tail'' of a collapsed graph.}
 \label{f.tail}
\end{figure}

\begin{remark}
 If one wants to see this graph as actually being a collapsing of the tree by identifying vertices on the same level which have the same labels, one should be careful as there are two ambiguities. First of all, for coherence purposes which will appear latter, we add a vertex $a+1$ on level $0$ which does not appear on the tree. Secondly, notice that for $n$ big enough, a collapsed graph on levels greater than $n$ is given in Figure~\ref{f.tail}. Such a ``tail'' of a collapsed graph is, \textit{a priori}, different from what one would get by identifying vertices on levels greater than $n$ on the tree. Indeed, there should not be edges between vertices labeled by $0$ on the collapsed graph. We deliberately chose to do it this way in order to control the length of the prefixes in $\mathcal{P}_{a,d}$. The reason for such a choice appears in the next subsection.
\end{remark}

Looking at this graph after collapsing is quite interesting since it gives us an understanding of the link between the summation process and the binary expansion of $a$. It also allows the statistical study of the behavior of $s_2(x+a)-s_2(x)$. An example of the tree $\tau_{25}$ collapsing is given in Figure~\ref{f.collapse}.

\subsection{Distribution of $s_2(x+a)-s_2(x)$}

The goal of this section is to prove the following theorem.

\begin{theorem}\label{thm.distrib}
 The distribution $\mu_a$ is calculated via 
an infinite product of matrices whose coefficients are operators of $l^1(\mathbb{Z})$:
$$ \mu_a = ( Id,\; Id ) \;  
    \cdots A_{a_n} A_{a_{n-1}} \cdots A_{a_1} A_{a_0} 
    \begin{pmatrix} \delta_0 \\ 0 \end{pmatrix},
$$
where the sequence $(a_n)_{n \in \mathbb{N}}$ is the binary expansion of $a$,  $\delta_0$ is the Dirac mass in $0$, the $A_i$ are defined by
$$
    A_0 = \begin{pmatrix} Id & \frac12 {S^{-1}} \\ 0 & \frac12 {S} \end{pmatrix}, \qquad 
    A_1 = \begin{pmatrix} \frac12 {S^{-1}} & 0 \\ \frac12 {S} & Id \end{pmatrix}, 
$$
and $S$ is the left shift transformation on $l^1(\mathbb{Z})$. 
\end{theorem}

This theorem follows from observations on the collapsed graph. First, notice that only two different patterns can appear between any two levels of the collapsed graph. We add labels on the edges of the collapsed graph in a way consistent with the tree $\tau_a$. Those patterns with the labels on edges are given by Figure~\ref{f.cells}. Notice also that the order in which the patterns appear is given by the binary expansion of $a$.  

Now, we remark that we can identify words in $\{0,1\}^*$ with paths in the collapsed graph whose starting point is the vertex labeled by $a$ on level $0$.

\begin{definition}
 We define the functions  $E:\{0,1\}^*\rightarrow \{1,2\}$ and $C:\{0,1\}^*\rightarrow \mathbb{Z}$ in the following way:
 \begin{itemize}
  \item $E(w)=1$ if the path associated to $w$ has endpoint the vertex of smaller label of its level (the leftmost vertex on Figure~\ref{f.cells});
  \item $E(w)=2$ if the path associated to $w$ has endpoint the vertex of greater label of its level (the rightmost vertex on Figure~\ref{f.cells});
  \item $C(w)$ is the ``counter'' of the path associated to $w$ (in the same way as for the tree defined in the previous subsection). The counter of the empty word is $0$ and the ``$+$'' and ``$-$'' on the edges of the graphs in Figure~\ref{f.cells} indicate whether to increment or decrement the counter. In other words, $C(w)$ is the number of edges of type ``$+$'' that the path associated to $w$ contains minus the number of edges of type ``$-$''.
 \end{itemize}

\end{definition}

Notice that the way the counter is defined is consistent with the counter on the tree $\tau_a$ for every integer $a$.


\begin{definition}

 Let us define, for each positive integer $a$, each level $n$ of the collapsed graph the probability measures $\eta_{a,n}^1$ and $\eta_{a,n}^2$ on every $d$ in $\mathbb{Z}$:
 
 $$
\eta_{a,n}^1(d)=\frac{1}{2^n}\#\{w \in \{0,1\}^n\ | \ E(w)=1 \ \text{and}\ C(w)=d\},
 $$
 
  $$
\eta_{a,n}^2(d)=\frac{1}{2^n}\#\{w \in \{0,1\}^n\ | \ E(w)=2 \ \text{and}\ C(w)=d\}
 $$
 
\end{definition}

This seems to only define a sequence in $l^1(\mathbb{Z})$, but we recall that we can identify finite measures on $\mathbb{Z}$ with sequences in $l^1(\mathbb{Z})$. It is a quick check to verify that we actually have  defined  probability measures.

Notice also that, for every $d$ in $\mathbb{Z}$, $\eta_{a,n}^1(d)$ is the asymptotic density of the set of integers whose binary expansion starts with a word in $\{w \in \{0,1\}^n\ | \ E(w)=1 \ \text{and}\ C(w)=d\}$.

\begin{figure}[ht]
\centering
\begin{tabular}{cccc}
 \includegraphics[scale=0.8]{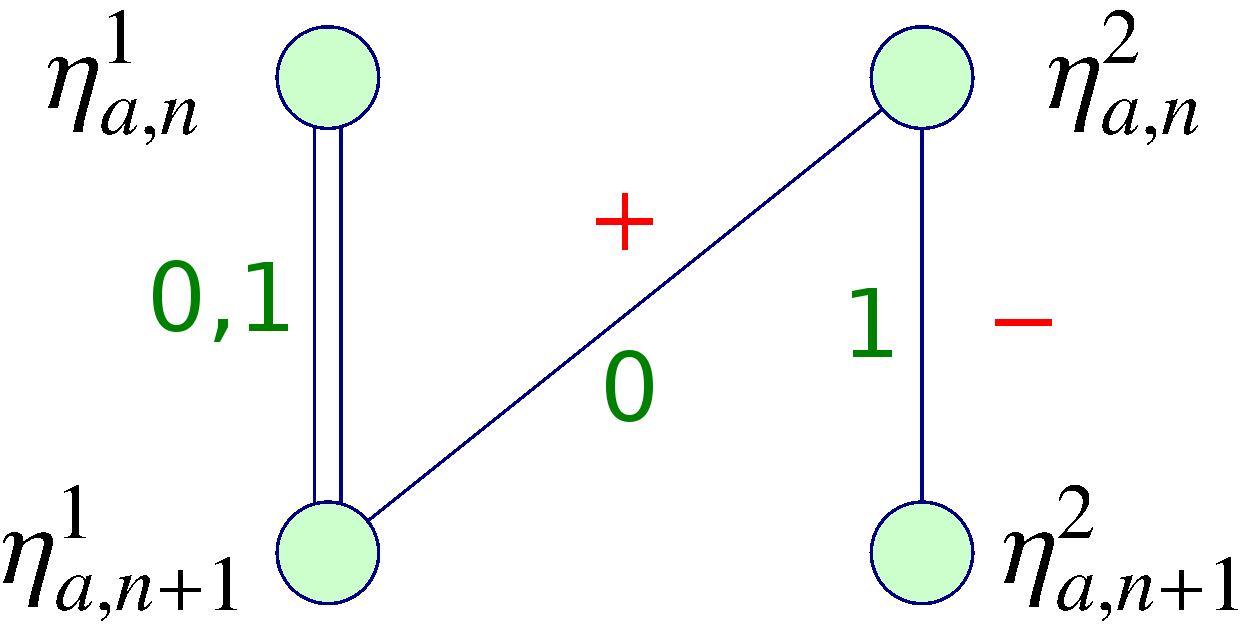}&&&
  \includegraphics[scale=0.8]{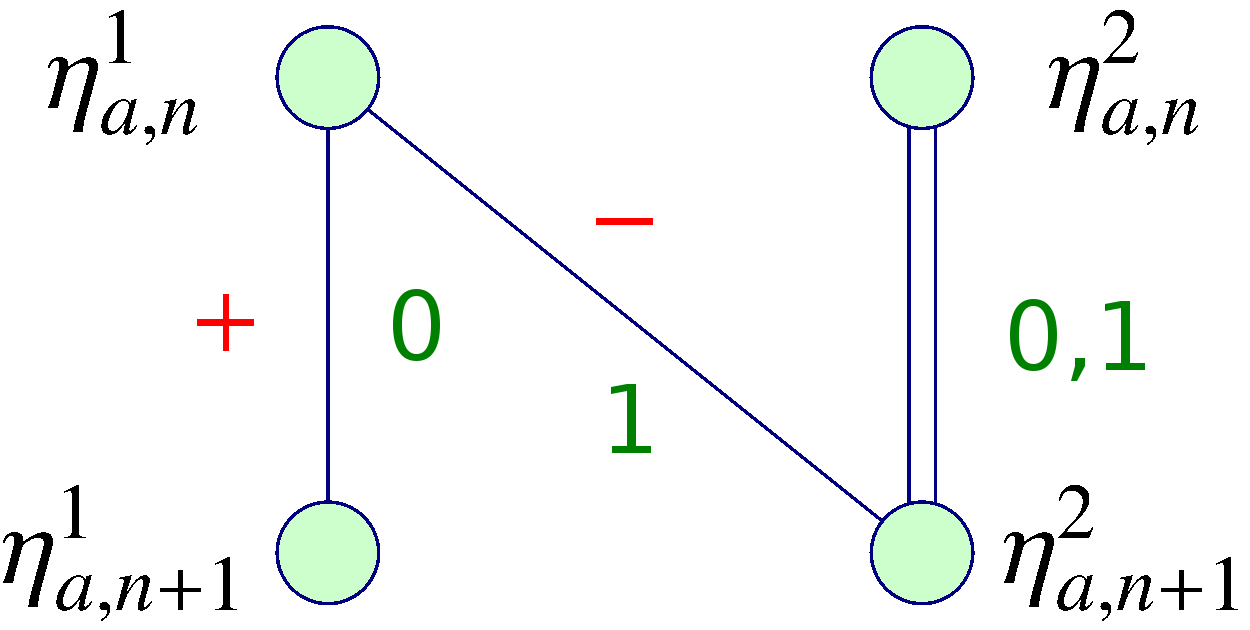}\\
\end{tabular}
 \caption{The two different patterns encountered in the collapsed graph}
 \label{f.cells}
\end{figure}

\begin{lemma}\label{l.eta_induction}
 The $\eta_{a,n}^i$ are given by induction:
 $$
 \begin{pmatrix} \eta_{a,n+1}^1 \\ \eta_{a,n+1}^2 \end{pmatrix} = A_{a_n} \begin{pmatrix} \eta_{a,n}^1 \\ \eta_{a,n}^2 \end{pmatrix}
 $$
 where 
 $$
 a=\sum_{k=0}^{+\infty} a_k  2^k.
 $$
\end{lemma}

\begin{proof}

We wish to compute $\eta_{a,n+1}^1$ and $\eta_{a,n+1}^2$. It is obvious that the pattern we see between levels $n$ and $n+1$ is given by $a_n$.
 
 If $a_n=0$ then we encounter the left pattern of Figure~\ref{f.cells}. In this case we have
 $$
 \eta_{a,n+1}^1(d)=\eta_{a,n}^1(d)+ \frac{1}{2}\eta_{a,n}^2(d-1).
 $$
 This being true for all $d$, we can write
 $$
 \eta_{a,n+1}^1=\eta_{a,n}^1+ \frac{1}{2} S^{-1}(\eta_{a,n}^2),
 $$
 where $S$ is the left shift transformation on the space $l^1(\mathbb{Z})$.
 For the same pattern we also have
 $$
 \eta_{a,n+1}^2(d)=\frac{1}{2}\eta_{a,n}^2(d+1)
 $$
 which can be rewritten as
 $$
 \eta_{a,n+1}^2=\frac{1}{2} S(\eta_{a,n}^2).
 $$
 So in this particular case we can write the relations in the following way:
 $$
 \begin{pmatrix} \eta_{a,n+1}^1 \\ \eta_{a,n+1}^2 \end{pmatrix} = \begin{pmatrix} Id & \frac12 {S^{-1}} \\ 0 & \frac12 {S} \end{pmatrix} \begin{pmatrix} \eta_{a,n}^1 \\ \eta_{a,n}^2 \end{pmatrix}.
 $$
 
 The same arguments for the pattern on the right of figure~\ref{f.cells} (which corresponds to the case where $a_n=1$) yields
 $$
 \begin{pmatrix} \eta_{a,n+1}^1 \\ \eta_{a,n+1}^2 \end{pmatrix} = \begin{pmatrix} \frac12 {S^{-1}} & 0 \\ \frac12 {S} & Id \end{pmatrix} \begin{pmatrix} \eta_{a,n}^1 \\ \eta_{a,n}^2 \end{pmatrix}.
 $$
\end{proof}

Lemma~\ref{l.eta_induction} allows us to prove Theorem~\ref{thm.distrib}:

\begin{proof}

We recall that for every $a$ in $\mathbb{N}$ and every $d$ in $\mathbb{Z}$, a set $\mathcal{P}_{a,d}$ is given by paths in the tree $\tau_a$ whose start is the root labeled by $a$ and whose end is a vertex labeled by $(0,d)$. Let $\mathcal{P}_{a,d}=\{p_1,...,p_{j}\}$ be this set of words, with $p_j$ being the prefix of maximal length.

Since we defined the `collapsed' graph in a way that is consistent with the tree $\tau_a$ (namely, the edges have the same labels), for every $i$ in $ \llbracket1,j\rrbracket$, the word $p_i$ represents a path starting from $a$ and ending on a vertex labeled by $0$ in the collapsed graph and satisfying $E(p_i)=1$ and $C(p_i)=d$.

We denote by $m_d$ the length of $p_j$, and consider the set of prefixes $\{p'_1,...,p'_k\} \subset \{0,1\}^{m_d}$ such that
$$
\bigcup_{i=1}^j [p_i]= \bigcup_{i=1}^k [p'_k].
$$
Notice that every element of  $\{p'_1,...,p'_k\}$ has a prefix that is an element of $\{p_1,...,p_{j}\}$.  Hence, for every $i$ in $\llbracket1,k\rrbracket$, we have $E(p'_i)=1$ and $C(p'_i)=d$ and thus:
$$
\{p'_1,...,p'_k\}\subseteq\{w \in \{0,1\}^{m_d}\ | \ E(w)=1 \ \text{and}\ C(w)=d\}
$$

Let us prove the reciprocal inclusion. Let $w$  be a word in $ \{0,1\}^{m_d}$ such that $E(w)=1$ and $C(w)=d$. Since $E(w)=1$, and since $w$ is of length $m_d$, its ending vertex is labeled by a $0$ (on the collapsed graph). Let us denote by $p$ the minimal prefix of $w$ such that the ending vertex of $p$ is a $0$. Then necessarily, $C(p)=d$. Now notice that such a path $p$ is also a path in $\tau_a$ with ending vertex labeled by $(0,d)$. Then $p$ is in $\{p_1,...,p_{j}\}$. So $w$ is in $\{p'_1,...,p'_k\}$. Consequently, 
$$
\{p'_1,...,p'_k\}=\{w \in \{0,1\}^{m_d}\ | \ E(w)=1 \ \text{and}\ C(w)=d\}
$$
and thus, by Remark~\ref{r.prefixes},
$$
\eta_{a,m_d}^1(d)=\mu_a(d).
$$

Finally, we have the following:
 $$
 \mu_a(d) = \left((Id,Id) A_{a_{m_d}} A_{a_{m_d-1}} \cdots A_{a_1} A_{a_0} \begin{pmatrix} \eta_{a,0}^1 \\ \eta_{a,0}^2 \end{pmatrix}\right)(d).
 $$
 
 It is also worth noting that, after rank $m_d$, continuing to multiply on the left by $A_0$ does not change the final result, so this allows for a more convenient expression:
 
 $$
 \mu_a(d) = \left((Id,Id)\cdots A_{a_{n}} A_{a_{n-1}} \cdots A_{a_1} A_{a_0} \begin{pmatrix} \eta_{a,0}^1 \\ \eta_{a,0}^2 \end{pmatrix}\right)(d).
 $$
 This proves the theorem.
\end{proof}

\begin{exmp}
 Let us show that Theorem~\ref{thm.distrib} allows us to find the expression of $\mu_1$ that was computed in the introduction. According to the theorem,
 $$
 \mu_1=(Id,Id) \cdots A_0 \cdots A_0 A_1 \begin{pmatrix}
                                          \delta_0\\
                                          0
                                         \end{pmatrix}.
 $$
 Now $A_1 \begin{pmatrix}
          \delta_0\\
          0
         \end{pmatrix} = \begin{pmatrix}
			  \frac12 \delta_1\\
			  \frac12 \delta_{-1}
			 \end{pmatrix}$ and it can  easily be seen that, for every integer $k$, $$A_0^k \begin{pmatrix}
			  \frac12 \delta_1\\
			  \frac12 \delta_{-1}
			 \end{pmatrix}=\begin{pmatrix}
			  \sum_{j=-1}^k \left(\frac12\right)^{j+2}\delta_{-j}\\
			  (\frac12)^k \delta_{-k-2}
			 \end{pmatrix}$$
whence

$$
\mu_a=\sum_{j=-1}^{\infty}\frac1{2^{j+2}} \delta_{-j}.
$$

Now, for a more involved example, let us assume we wish to compute $\mu_3(0)$.
 
 Well,
 $$
 \mu_3=(Id,Id) \cdots A_0 \cdots A_0 A_1 A_1 \begin{pmatrix}
                                          \delta_0\\
                                          0
                                         \end{pmatrix},
 $$
 then a quick computation yields
 $$A_1 A_1 \begin{pmatrix}
          \delta_0\\
          0
         \end{pmatrix} = \begin{pmatrix}
			  \frac14 \delta_2\\
			  \frac14 \delta_0 + \frac12 \delta_{-1}
			 \end{pmatrix}
 $$
 and so
 $$A_0 A_1 A_1 \begin{pmatrix}
          \delta_0\\
          0
         \end{pmatrix} = \begin{pmatrix}
			  \frac14 \delta_2 + \frac18 \delta_1 + \frac14 \delta_0\\
			  \frac18 \delta_{-1} + \frac14 \delta_{-2}
			 \end{pmatrix}.
 $$
 Now in order to know $\mu_3(0)$, we have to multiply by $A_0$ once more: 
 $$A_0 A_0 A_1 A_1 \begin{pmatrix}
          \delta_0\\
          0
         \end{pmatrix} = \begin{pmatrix}
			  \frac14 \delta_2 + \frac18 \delta_1 + \frac{5}{16} \delta_0+\frac18 \delta_{-1}\\
			  \frac1{-16} \delta_{-2} + \frac18 \delta_{-3}
			 \end{pmatrix}.
 $$
 Note that multiplying by $A_0$ again will never change the value of the coefficient of $\delta_0$ in the first coordinate, neither will $\delta_0$ appear in the second coordinate (since we apply $\frac12 S$ to the second coordinate when multiplying by $A_0$). Hence $\mu_3(0)=\frac{5}{16}$.
 
 In the general case, for given $a$ in $\mathbb{N}$ and $d$ in $\mathbb{Z}$, we should first compute $A_{a_n}\cdots A_{a_0} \begin{pmatrix}
          \delta_0\\
          0
         \end{pmatrix}$, which gives a vector whose coordinates are linear combinations of Dirac measures. Then we should multiply by $A_0$ until the Dirac measures on the second coordinate have mass only on integers no greater than $d-2$, which ensures that the coefficient of $\delta_d$ in the first coordinate is exactly $\mu_a(d)$.
\end{exmp}

\section{Asymptotic properties of distributions $\mu_a$}

In this section we study the asymptotic behavior of the family of distributions $\mu_a$ as $l(a)$ increases.
We prove that the $l^2(\mathbb{Z})$ norm of $\mu_a$ tends to $0$ as $l(a)$ increases, and  hence the densities of~$\mu_a$ tend to zero. 

Let us recall that $l(a)$ denotes the number of distinct patterns $01$ in the binary expansion of~$a$.
\begin{remark}
For every integer $a$ there exists $k$ integers $p_1,...,p_k$ such that $\underline{a}=0^{p_1}1^{p_2}...1^{p_k}$ or $\underline{a}=1^{p_1}0^{p_2}...1^{p_k}$ depending on the parity of $a$. If $a$ is even, $k=2l(a)$, and else, $k=2l(a)+1$.

In other words, the number of distinct continuous maximal blocks of digits in $\underline{a}$ is either $2l(a)+1$ or $2l(a)$.
\end{remark}

\subsection{Convergence to zero} 

We start by proving the following theorem.
%
\begin{theorem}\label{th.asympt}
There exists a constant $C_0$ such that for every integer $a$ we have the following:
$$
  \|\mu_a\|_2 \le C_0 \cdot l(a)^{-1/4}.
$$
\end{theorem}

The proof is established in a series of lemmas. Let us recall the notation 
$$
    A_0 = \begin{pmatrix} Id & \frac12 {S^{-1}} \\ 0 & \frac12 {S} \end{pmatrix}, \qquad 
    A_1 = \begin{pmatrix} \frac12 {S^{-1}} & 0 \\ \frac12 {S} & Id \end{pmatrix}, 
$$
where $S$ is the left shift on~$l^1(\mathbb{Z})$. 
Let us define a norm on the space of $n \times n$ matrices $Y = (y_{i,j})_{i,j \in \{1,...,n\}^2}$ by the following:
$$
  \|Y\|_1 = \max_{1 \le j \le n} \sum_{i=1}^n |y_{i,j}|. 
$$

\begin{remark}
 We remark right away that this defines a submultiplicative norm.
\end{remark}

We consider the Fourier transform $\Hat\mu_a$ defined for every $\theta$ in $[0,2\pi)$ by
$$
\Hat\mu_a(\theta)=\sum_{d\in \mathbb{Z}} e^{-id\theta}\mu_a(d).
$$
Notice that, given an element $f$ of $l^1(\mathbb{Z})$, for every $\theta$ in $[0,2\pi)$, $\Hat{S(f)}(\theta)=e^{i\theta}\Hat{f}(\theta)$. Hence, since
$$
 \mu_a(d) = \left((Id,Id)\cdots A_{a_{n}} A_{a_{n-1}} \cdots A_{a_1} A_{a_0} \begin{pmatrix} \delta_0 \\ 0 \end{pmatrix}\right)(d)
 $$
 we get that
 $$
 \Hat{\mu_a}(\theta) = (1,1)\cdots \Hat{A}_{a_{n}}(\theta) \Hat{A}_{a_{n-1}}(\theta) \cdots \Hat{A}_{a_1}(\theta) \Hat{A}_{a_0}(\theta) \begin{pmatrix} 1 \\ 0 \end{pmatrix}
 $$
 where
$$
  \hA_0(\theta):= \begin{pmatrix} 1 & \frac12 e^{-i \theta} \\ 0 & \frac12 e^{i \theta} \end{pmatrix}, \qquad 
   \hA_1(\theta):= \begin{pmatrix} \frac12 e^{-i \theta} & 0 \\ \frac12 e^{i \theta} & 1 \end{pmatrix}, 
$$

\begin{lemma}\label{l.phi}
We have the following elementary identities for every $\theta$:
\begin{itemize}
 \item $\|\hA_0(\theta)\|_1 = \|\hA_1(\theta)\|_1 = \|\hA_0(\theta)\hA_1(\theta)\|_1 = \|\hA_1(\theta)\hA_0(\theta)\|_1 = 1.$
 \item$  \|\hA_0(\theta)\hA_1(\theta)\hA_0(\theta)\|_1 = \|\hA_1(\theta)\hA_0(\theta)\hA_1(\theta)\|_1 
 := \phi(\theta) = \frac{1 + \sqrt{5+4\cos\theta}}4.$
\end{itemize}
\end{lemma}

Notice that $\phi$ is strictly less than~$1$ except when $\theta = 0$. 

\begin{proof}
For every $\theta$ in $[0,2\pi)$, we have
$$
\hA_0(\theta)\hA_1(\theta)\hA_0(\theta)= 
\begin{pmatrix}
\frac12e^{-i\theta}+\frac14 & \frac14e^{-2i\theta}+\frac18e^{-i\theta}+\frac14 \\
\frac14e^{2i\theta} & \frac18e^{i\theta} + \frac14e^{2i\theta}
\end{pmatrix},
$$
and
$$
\hA_1(\theta)\hA_0(\theta)\hA_1(\theta)= 
\begin{pmatrix}
 \frac18e^{-i\theta} + \frac14e^{-2i\theta}   &   \frac14e^{-2i\theta}   \\
 \frac14e^{2i\theta}+\frac18e^{i\theta}+\frac14   & \frac12e^{i\theta}+\frac14
\end{pmatrix}.
$$
So we have 
$$
\|\hA_0(\theta)\hA_1(\theta)\hA_0(\theta)\|_1 = \|\hA_1(\theta)\hA_0(\theta)\hA_1(\theta)\|_1 .
$$
Moreover, the triangular inequality yields
$$
\left|\frac14e^{-2i\theta}+\frac18e^{-i\theta}+\frac14\right| + \left|\frac14e^{2i\theta}+\frac18e^{i\theta}\right| \leq \left|\frac14e^{-2i\theta}\right|+\left|\frac18e^{-i\theta}+\frac14\right| + \left|\frac14e^{2i\theta}+\frac18e^{i\theta}\right|,
$$
so we have 
$$
\left|\frac14e^{-2i\theta}+\frac18e^{-i\theta}+\frac14\right| + \left|\frac14e^{2i\theta}+\frac18e^{i\theta}\right| \leq \left| \frac14e^{2i\theta}\right|  + \left|  \frac12e^{-i\theta} +\frac14  \right|.
$$
Hence the sum of the moduli of the terms in the first column is less than the sum of the moduli of the terms of the second column, which implies that
$$
\|\hA_0(\theta)\hA_1(\theta)\hA_0(\theta)\|_1 = \left| \frac14e^{-2i\theta}\right|  + \left|  \frac12e^{i\theta} +\frac14  \right|
$$
so finally
$$
\|\hA_0(\theta)\hA_1(\theta)\hA_0(\theta)\|_1=\frac{1 + \sqrt{5+4\cos\theta}}4.
$$
\end{proof}

\begin{lemma}\label{l.reductionto010}
For every $k$ in $\mathbb{N}^*$ and every $\theta$ in $[0,2\pi)$ 
$$
  \|\hA_0(\theta)\, (\hA_1(\theta))^k \,\hA_0(\theta)\|_1 \le 
  \|\hA_0(\theta)\, \hA_1(\theta) \,\hA_0(\theta)\|_1 .
$$ 
\end{lemma}

\begin{proof}
First let us state that, for all positive integer $k$, for all $\theta$ in $[0,2\pi)$,
$$
(\hA_1(\theta))^k=\begin{pmatrix} 
                                                      \frac{e^{-ik\theta}}{2^k} & 0 \\
                                                      \sigma_k(\theta) & 1
                                                     \end{pmatrix}
$$
where
\begin{equation}\label{sigma}
 \sigma_k(\theta)=e^{2i\theta}\sum_{j=1}^k \frac{e^{-ij\theta}}{2^j}.
\end{equation}

We can compute that, for every $k$ in $n \mathbb{N}$
$$
\quad \hA_0(\theta)\, (\hA_1(\theta))^k \,\hA_0(\theta)=
\begin{pmatrix}
\frac{e^{-ik\theta}}{2^k} + \frac{e^{-i \theta}}2 \sigma_k(\theta) \quad& \frac{e^{-i(k+1)\theta}}{2^{k+1}} +\frac{e^{-2i \theta}}4 \sigma_k(\theta) + \frac14 \\
&\\
\frac{e^{i \theta}}2 \sigma_k(\theta) & \frac14 \sigma_k(\theta) + \frac{e^{2i \theta}}4
\end{pmatrix}.
$$
Moreover, noticing that for every positive integer $k$ and every $\theta$ in $[0,2\pi)$, 
$$
\sigma_{k+1}(\theta)=\frac{e^{-i\theta}}2 \sigma_k(\theta)+\frac{e^{i\theta}}2
$$
implies that
\begin{align*}
&\left|\frac{e^{-i(k+1)\theta}}{2^{k+1}} + \frac{e^{-i \theta}}2 \sigma_{k+1}(\theta)\right| + \left|\frac{e^{i \theta}}2 \sigma_{k+1}(\theta)\right|\\
=&
\left| \frac{e^{-i(k+1)\theta}}{2^{k+1}} +\frac{e^{-2i \theta}}4 \sigma_k(\theta) + \frac14 \right| + \left|\frac14 \sigma_k(\theta) + \frac{e^{2i \theta}}4\right|
.\end{align*}

This means that the sum of moduli of coefficient on the first column of matrix $\hA_0(\theta)\, (\hA_1(\theta))^{k+1} \,\hA_0(\theta)$ is equal to the sum of the moduli of the coefficients of the second column of the matrix $\hA_0(\theta)\, (\hA_1(\theta))^k \,\hA_0(\theta)$.
Hence, to prove that, for every integer $k$ and all $\theta$ in $[0,2\pi)$, 
$$  \|\hA_0(\theta)\, (\hA_1(\theta))^k \,\hA_0(\theta)\|_1 \le 
  \|\hA_0(\theta)\, \hA_1(\theta) \,\hA_0(\theta)\|_1 
$$ 
it is enough to actually show the following:
$$
\left|\frac{e^{-ik\theta}}{2^k} + \frac{e^{-i \theta}}2 \sigma_k(\theta)\right| + \left|\frac{e^{i \theta}}2 \sigma_k(\theta)\right| \leq \|\hA_0(\theta)\, \hA_1(\theta) \,\hA_0(\theta)\|_1.
$$
Using the triangle inequality, we have
$$
\left|\frac{e^{-ik\theta}}{2^k} + \frac{e^{-i \theta}}2 \sigma_k(\theta)\right| + \left|\frac{e^{i \theta}}2 \sigma_k(\theta)\right|\leq \frac1{2^k}+\left|\sigma_k(\theta)\right|
$$
which, still using the triangle inequality, again gives
$$
 \left|\frac{e^{-ik\theta}}{2^k} + \frac{e^{-i \theta}}2 \sigma_k(\theta)\right| + \left|\frac{e^{i \theta}}2 \sigma_k(\theta)\right|\leq \frac1{2^k}+\left|\frac{e^{i\theta}}2+\frac14\right|+\sum_{j=3}^k\frac1{2^j}
$$
and
$$ \left|\frac{e^{-i2\theta}}{2^2} + \frac{e^{-i \theta}}2 \sigma_2(\theta)\right| + \left|\frac{e^{i \theta}}2 \sigma_2(\theta)\right|\leq \frac1{2^2}+\left|\frac{e^{i\theta}}2+\frac14\right|.
$$
So we get
$$
\left|\frac{e^{-ik\theta}}{2^k} + \frac{e^{-i \theta}}2 \sigma_k(\theta)\right| + \left|\frac{e^{i \theta}}2 \sigma_k(\theta)\right|\leq \left|\frac{e^{i\theta}}2+\frac14\right|+\frac14
$$
which completes the proof since 
$$
\left|\frac{e^{i\theta}}2+\frac14\right|+\frac14= \|\hA_0(\theta)\, \hA_1(\theta) \,\hA_0(\theta)\|_1.
$$
\end{proof}

\begin{lemma}\label{l.gauss}
If $-\pi \le \theta \le \pi$ then 
$$
  \phi(\theta) \le e^{-\theta^2/15}. 
$$

Hence, for every positive integer $N$,
$$
\displaystyle \|\phi^N\|_2 \le \|e^{-N\theta^2/15}\|_2 = \left( \frac{15\pi}{2N} \right)^{1/4}
$$

with $\phi$ defined as in Lemma~\ref{l.phi}.
\end{lemma}
\begin{proof}
 The proof is left to the reader.
\end{proof}

\begin{proof}[Proof of Theorem~\ref{th.asympt}]
%
We remind the reader that $\mu_a = \mu_{a,1} + \mu_{a,2}$ is the sum of the components 
of the vector $\bar\mu_a$, which is calculated as an infinite product 
$$
  \bar\mu_a = 
  \begin{pmatrix} \mu_{a,1} \\ \mu_{a,2} \end{pmatrix} = 
  \cdots A_{a_n} A_{a_{n-1}} \cdots A_{a_1} A_{a_0} 
  \begin{pmatrix} \delta_0 \\ 0 \end{pmatrix}. 
$$
Let us represent the binary expansion of $a$ as a sequence 
of $l(a)$ groups of $11\ldots1$ separated by zeros: 
$$
  \underline{a} = 
  0^{m_1}\underline{11\ldots1}0^{m_2}\underline{11\ldots1}0\dots\dots\dots  
         0^{m_{l(a)}}\underline{11\ldots1}. 
$$
We apply Lemma~\ref{l.reductionto010} to half of the patterns $011..10$ and then use the bound of
Lemma~\ref{l.phi}:
$$
  \|\hA_0(\theta)\, (\hA_1(\theta))^k \,\hA_0(\theta)\|_1 \le \phi(\theta). 
$$

We can do this only on half of $011..10$'s because we need to avoid problematic patterns like $\ldots0111{\mathbf 0}11110\ldots$, 
since we need at least two ``0'' between two blocks on ``1'' to be able to apply Lemma~\ref{l.reductionto010} twice but we have only one ``0'' in between. Thus we get 
$$
  \|\cdots \hA_{a_n}(\theta) \cdots \hA_{a_0}(\theta)\|_1 \le \phi(\theta)^{N}, 
  \quad N = \frac{l(a)-1}2, 
$$
and hence, for each $\theta$ in $[0,2\pi)$ and $j$ in $\{1,2\},$
$$
  |\hat\mu_{a,j}(\theta)| \le 
    \phi(\theta)^{N}  \left\| \vv{1}{0} \right\|_1 = e^{-N\theta^2/15},  
$$
and
$$
  \|\mu_{a,j}\|_2 = \frac1{\sqrt{2\pi}} \|\hat\mu_{a,j}\|_2 \le 
  \frac1{\sqrt{2\pi}} \|\phi^N\|_2 \leq 
  \frac1{\sqrt{2\pi}} \sqrt{\int_{-\infty}^\infty e^{-2Nt^2/15} \,dt }= \left( \frac{15}{8\pi N} \right)^{1/4}
$$
%
by Lemma~\ref{l.gauss}.

Now, $ N = \frac{l(a)-1}2$ yields
$$
  \|\mu_a\|_2 \le \sqrt{2} \left( \frac{15}{\pi(l(a)-1)} \right)^{1/4} 
    = O\!\left(l(a)^{-1/4}\right). 
$$
\end{proof}

\begin{remark}
It follows directly from Theorem~\ref{th.asympt} 
that the density $\mu_a(j) \to 0$ as $l(a) \to \infty$. 

Let us also remark that this theorem actually gives important information concerning the fact that the probability measure $\mu_a$ is linked with the complexity of the binary expansion of $a$ (which is measured by $l(a)$). In fact, for every integer $n$, $\mu_{2^n}=\mu_1$, so it is possible to find arbitrarily large $a$ such that $||\mu_a||_2$ does not tend to zero (actually whenever $l(a)$ does not tend to infinity).
\end{remark}
\subsection{Asymptotic mean and variance of $\mu_a$}
Let us first recall the following:
$$
  \hA_0(\theta):= \begin{pmatrix} 1 & \frac12 e^{-i \theta} \\ 0 & \frac12 e^{i \theta} \end{pmatrix}, \qquad 
   \hA_1(\theta):= \begin{pmatrix} \frac12 e^{-i \theta} & 0 \\ \frac12 e^{i \theta} & 1 \end{pmatrix}, 
$$
for every $\theta$ in $[0,2\pi)$.

We begin by using Taylor's expansion on the matrices $\hA_0(\theta)$ and $\hA_1(\theta)$ to get
$$
  \hA_k(\theta) = I_k + \theta \,\alpha_k + \theta^2 \,\beta_k + O(\theta^3), \qquad k \in \{0,\, 1\}, 
$$
where 
\begin{alignat*}{3}
  &
  I_0 = \mm{1}{\frac12}{0}{\frac12}, \quad\; && 
  \alpha_0 = \frac{i}2 \mm{0}{-1}{0}{1}, \quad\; && 
  \beta_0 = -\frac14 \mm{0}{1}{0}{1}, 
  \\
  &
  I_1 = \mm{\frac12}{0}{\frac12}{1}, \quad\; && 
  \alpha_1 = \frac{i}2 \mm{-1}{0}{1}{0}, \quad\; && 
  \beta_1 = -\frac14 \mm{1}{0}{1}{0}, 
\end{alignat*}
and we observe that the following relations hold
\begin{gather}
  (1,1)I_k = (1,1), \qquad 
  \\
  (1,1)\alpha_k = (0,0), \qquad 
  (1,1)\beta_0 = -\frac12(0,1), \qquad 
  (1,1)\beta_1 = -\frac12(1,0). 
\end{gather}

\begin{lemma}
For all $\theta$ in $[0, 2\pi)$, the product of matrices $(\hA_0(\theta))^k$ converges as $k$ goes to $+\infty$. We denote the limit $ \hA_0^\infty(\theta)$, and further we have the following equality:
$$
  \hA_0^\infty(\theta) = \mm{ 1 }{  \frac{ e^{-i\theta} }{ 2-e^{i\theta} }  }{ 0 }{ 0 }. 
$$
\end{lemma}

\begin{proof}
We recall that $\sigma_k$ is defined in Section 3.1, Equation~(\ref{sigma}).

The lemma is proved by remarking the following:
$$
(\hA_0(\theta))^k=\begin{pmatrix} 
                                                      1 & \overline{\sigma_k(\theta)} \\
                                                      0 & \frac{e^{ik\theta}}{2^k}
                                                     \end{pmatrix}.
$$
\end{proof}

Now the asymptotic expansion of $\hA_0^\infty$ is

$$\hA_0^\infty(\theta)=I_{\infty} + \theta^2 \beta_{\infty} + O(\theta^3), $$

where

$$
 I_{\infty}=\mm1100, \quad \beta_{\infty} = - \mm0100.$$

\begin{definition}
Given $a$ in $\mathbb{N}$ of binary length $N$, we define the following product: $$\Pi_a(\theta) = \hA_0^\infty \hA_{a_N}(\theta) \cdots \hA_{a_0}(\theta)$$.  

\end{definition}

Let us recall that, for every $a$ in $\mathbb{N}$,

$$\Hat{\mu_a}(\theta)=(1,1) \cdot \Pi_a(\theta) \cdot \begin{pmatrix}
                                                      1 \\
                                                      0
                                                     \end{pmatrix}.
$$

\begin{lemma}\label{l.taylor}

Let $v_0=\begin{pmatrix}
                                                      1 \\
                                                      0
                                                     \end{pmatrix}.$ The product $(1,1) \cdot \Pi_a(\theta) \cdot v_0$ has the following asymptotic expansion:
$$
  (1,1) \cdot \Pi_a(\theta) \cdot v_0 = 1 + V(a)\,\theta^2 + O(\theta^3), 
$$
where
$$
  V(a) = (1,1)\left(\beta_{\infty} v_{N+1} + \beta_{a_N} v_N + \ldots + \beta_{a_0} v_0\right), 
$$
and, for all $j$ in $\{ 1,..., N+1 \}$,
$$
v_j = I_{a_{j-1}} \cdots I_{a_0} v_0. 
$$
\end{lemma}

In other words, this Lemma~\ref{l.taylor} states that $\mu_a$ has mean $0$ and variance $-2V(a)$. Indeed, the moments of a probability measure are given by the successive derivatives near 0 of its Fourier transform.

\begin{proof}
Let us first recall that

$$
(1,1)I_0=(1,1)I_1=(1,1)I_\infty=(1,1).
$$

Moreover,
$$
(1,1) \cdot \Pi_a(\theta) \cdot v_0=(1,1)\left(I_{\infty} + \theta^2 \beta_{\infty} + O(\theta^3)\right)\prod_{i=0}^N\left(I_{a_i}+\alpha_{a_i}\theta + \beta_{a_i}\theta^2 + O(\theta^3)\right)v_0,
$$
where the product is defined as:
$$
\prod_{i=0}^N M_i := M_N...M_0.
$$
Hence the constant term in the asymptotic expansion of $ (1,1) \cdot \Pi_a(\theta) \cdot v_0$ is $1$ (since $(1,1) \begin{pmatrix}
                                                                                                                    1\\0
                                                                                                                   \end{pmatrix}
=1$). In addition, we have
$$
(1,1)\alpha_0=(1,1)\alpha_1=(0,0),
$$
so the term of degree 1 is 0.

With the same argument we can compute the quadratic term $V(a)$.
Notice that every coefficient $\alpha_0$ or $\alpha_1$ is killed by left multiplication by $(1,1)$.
Hence the terms of the form
$$
I_{\infty}...I_{a_{j+1}}\alpha_{a_j}I_{a_{j-1}}...I_{a_{i+1}}\alpha_{a_i}I_{a_{i-1}}...I_{a_0}
$$
disappear after  left multiplication by $(1,1)$, and thus do not contribute in any way to the coefficient of the quadratic term.

So the only terms contributing to $V(a)$ in the product $\hA_\infty(\theta) \hA_{a_N}(\theta)...\hA_{a_0}(\theta)$ are the terms with coefficient $\beta_{a_j}$ with $j$ in $\llbracket0,N\rrbracket$ and the term with coefficient $\beta_\infty$.
\end{proof}

\begin{definition}
Let $a$ be a positive integer with binary expansion $\underline{a}=a_0...a_n$. For every $j\in \llbracket 0,n\rrbracket$, let
$$
b^{1}_j:=\max \left\{k+1\ | \ a_{j-k}=...=a_j\right\}
$$
and
$$
b^{2}_j:=b^{1}_{j-b^{1}_j}.
$$

\end{definition}

In other words, $b^{1}_j$ is the length of a consecutive and maximal sequence of digits equal to $a_j$ to the left of $j$ (including position $j$) in $\underline{a}$ and $b^{2}_j$ is the length of the next block of identical digits to the left of the block containing the digit~$a_j$.

For instance, if $\underline{a}=100111110011$, then $b^1_6=4$ since $a_2...a_6=01111$ and $b^2_6=2$.

\begin{lemma}\label{l.sum}
For every $j$ in $\{0,...,N\}$, the value $(1,1)\beta_{a_j} v_j$ is estimated as follows 
$$
  \frac1{2^{b^{1}_j}} \left( 1 - \frac1{2^{b^{2}_j}} \right) \le 
  -(1,1)\beta_{a_j} v_j \le 
  \frac1{2^{b^{1}_j}}, 
$$
\end{lemma}

\begin{proof}
Assume that $a_{j} = 0$ (the proof for the case $a_{j} = 1$ is similar). 
Observe that $I_0$ and $I_1$ contract the segment $[(1,0),(0,1)]$  
to the first  second point respectively with a contraction factor of $2$. 

More precisely, 
$$
\left\{I_0\begin{pmatrix}
           1-x\\
           x
          \end{pmatrix}\ |\ x \in [0,1] \right\}=\left\{\begin{pmatrix}
           1-x\\
           x
          \end{pmatrix}\ |\ x \in \left[0,\frac12\right] \right\}
$$
and
$$
\left\{I_1\begin{pmatrix}
           x\\
           1-x
          \end{pmatrix}\ |\ x \in [0,1] \right\}=\left\{\begin{pmatrix}
           x\\
           1-x
          \end{pmatrix}\ |\ x \in \left[0,\frac12\right] \right\}.
$$
Hence, with $t=j-b_j^2-b_j^1$,
$$
I_t...I_0 \begin{pmatrix}
           1\\0
          \end{pmatrix} \in [(1,0),(0,1)].
$$
Moreover, the block $I_1^{b^2_j}$ maps $[(1,0),(0,1)]$ to the segment $[(2^{-b^2_j},(1-2^{-b^2_j})),(0,1)]$. 
Then the block $I_0^{b^1_j-1}$ contracts the latter to 
$$
  \Bigl[
    \bigl(
      1 - 2^{-b^1_j+1} (1-2^{-b^2_j}),   2^{-b^1_j+1} (1-2^{-b^2_j})
    \bigr),\; 
    \bigl(
      1 - 2^{-b^1_j+1},   2^{-b^1_j+1} 
    \bigr)
  \Bigr].
$$
Multiplying on the left by $(1,1)\beta_{a_j} = (1,1)\beta_0=-\frac12(0,1)$ means taking the second coordinate 
with an~additional coefficient $-1/2$, and the lemma then follows. 
\end{proof}

\begin{remark}
 The same proof yields
 
 $$
 -(1,1)\beta_{\infty} v_{N+1} \le 
  \frac1{2^{b^{1}_{N}-1}}.
 $$
 
\end{remark}

Moreover, we have the following lemma.
\begin{lemma}\label{lemma.minor1}
 $$
l(a)\leq \sum_{j=0}^N \frac1{2^{b^1_j}} \le 2l(a)+1.
$$
\end{lemma}

\begin{proof}
This lemma is obtained by noticing that each block of identical digits in the binary expansion of $a$ spans a term $\sum_{j=0}^N \frac1{2^{b^1_j}}$ which is a partial sum of a geometric sequence of ratio $\frac12$ and first term $\frac12$, so each is smaller than 1 and greater than $\frac12$. In fact, let us assume that there are $k$ blocks in the binary expansion of $a$ of lengths $p_1,...,p_k$ (i.e $\underline{a}=0^{p_1}1^{p_2}...0^{p_{k-1}}1^{p_k}$, for example). Then
$$
\sum_{j=0}^N \frac1{2^{b^1_j}}=\sum_{i=1}^{k}\sum_{j=1}^{p_i}\left(\frac12\right)^{j}
$$
and 
$$
\sum_{i=1}^{k} \frac12 \leq \sum_{i=1}^{k}\sum_{j=1}^{p_i}\left(\frac12\right)^{j} \leq \sum_{i=1}^{k} 1.
$$
Moreover, $k$ can be equal to either $2l(a)+1$ or $2l(a)$, which completes the proof.
\end{proof}

\begin{theorem}\label{thm.var}
For every integer $a$, the variance $-2V(a)$ of $\mu_a$ has bounds
 
 $$
 l(a)-1 \leq -2V(a) \leq 2(2 l(a)+2).
 $$
\end{theorem}

\begin{proof}
The idea of this theorem is to see that each block of same digits 
in the binary expansion of~$a$ has an~impact on $V(a)$ estimated by a~constant.

Let us estimate the value of $V(a)$. By Lemma~\ref{l.sum} we have that
$$
  -V(a) = -\sum_{j=0}^N (1,1)\beta_{a_j} v_j -(1,1)\beta_{\infty} v_{N+1} \le \sum_{j=0}^N \frac1{2^{b^1_j}}+\frac1{2^{b^{1}_{N}-1}}.
$$
 Moreover, the upper bound of Lemma~\ref{lemma.minor1} yields
$$-2V(a) \leq 2(2 l(a)+2).$$

Now we will prove the lower bound:
$$
  -V(a) \ge \frac12(l(a)-1). 
$$
Note that
$$
  -V(a) = -\sum_{j=0}^N (1,1)\beta_{a_j} v_j - (1,1)\beta_{\infty}v_{N+1} \geq  -\sum_{j=0}^N (1,1)\beta_{a_j} v_j
$$
so, using Lemma~\ref{l.sum}, we have:
$$
-V(a) \geq \sum_{j=0}^N\left(\frac1{2^{b^{1}_j}} \left( 1 - \frac1{2^{b^{2}_j}} \right)\right).
$$
Whence, 
$$
-V(a) \geq \sum_{j=0}^N\frac1{2^{b^{1}_j}} -\sum_{j=0}^N \frac1{2^{b^{1}_j+b^{2}_j}}.
$$
Let us denote by $k$ the largest integer such that $a_0=...=a_k$.
Noticing that $b^{2}_j=0$ for all $j$ in $\{ 0,..., k\}$, we have that
$$
-V(a) \geq \sum_{j=k+1}^N\frac1{2^{b^{1}_j}} -\sum_{j=k+1}^N \frac1{2^{b^{1}_j+b^{2}_j}}.
$$
Moreover, for every integer $j$ larger than $k$, $b^{2}_j\geq 1$, hence
$$
-V(a) \geq \frac12 \sum_{j=k+1}^N\frac1{2^{b^{1}_j}} 
$$
from which we derive
$$
-V(a) \geq \frac12 (l(a)-1)
$$
by applying the lower bound of Lemma~\ref{lemma.minor1}. This proves Theorem~\ref{thm.var}. 
\end{proof}

On a final note, it is interesting to remark that the bounds in Theorem~\ref{thm.var} might not be optimal but that for a sequence of integers $(a(n))_{n\in \mathbb{N}}$ such that $\lim_{n\rightarrow \infty} l(a(n))=+\infty$ and such that the limit of $\frac{-2V(a(n))}{l(a(n))}$ exists, then this limit can take different values in the interval $[1,4]$. We give some examples of computer simulations in Figure~\ref{f.ratio}.

It would be interesting to understand the asymptotic behavior of $\frac{-2V(a(n))}{l(a(n))}$. Having a necessary condition for convergence and a precise idea of how it behaves asymptotically would help in understanding the variance of the probability measure $\mu_a$.

\newpage

\begin{figure}[ht]
\subfloat[$\frac{-2V(a(n))}{l(a(n))}$ for $a(n)=\sum_{k=0}^n 2^{2k}$]{
 \includegraphics[scale=0.3]{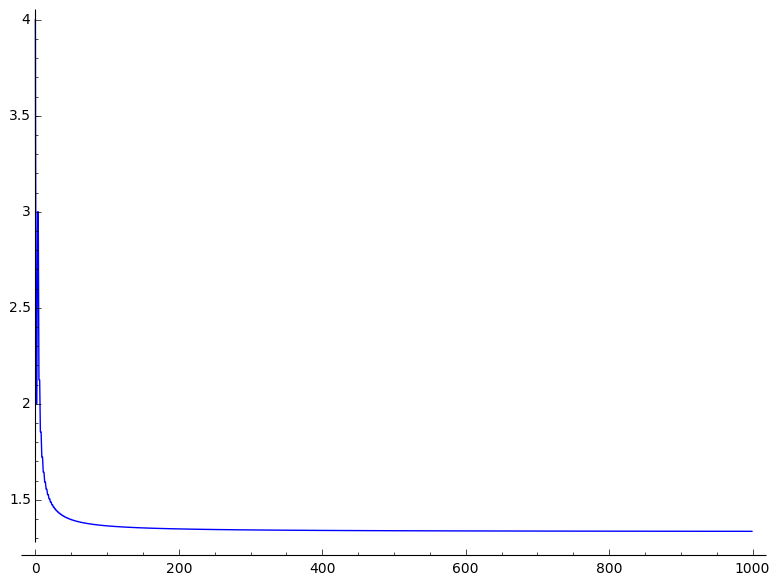}
}
\subfloat[$\frac{-2V(a(n))}{l(a(n))}$ for $a(n)=\sum_{k=0}^n 2^{4k}+2^{4k+1}$]{
 \includegraphics[scale=0.3]{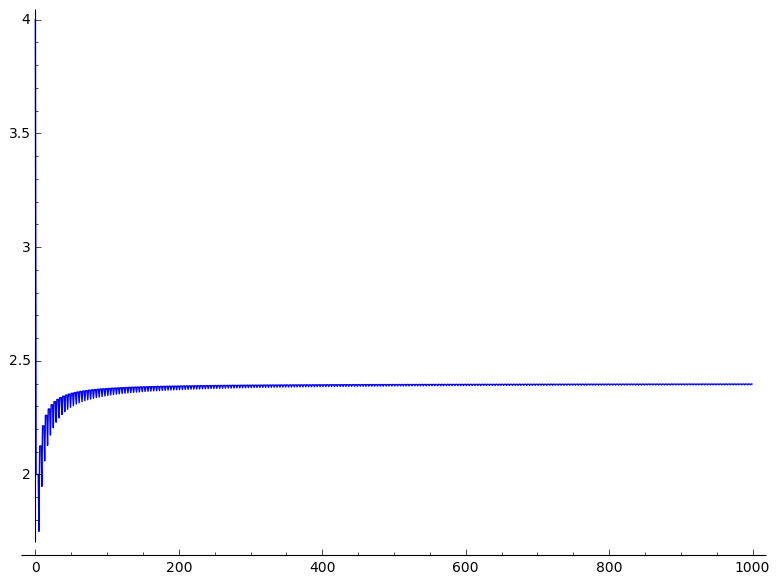}
 }

\subfloat[$\frac{-2V(a(n))}{l(a(n))}$ for $a(n)=\sum_{k=0}^n 2^{6k}+2^{6k+1}+2^{6k+2}$]{
 \includegraphics[scale=0.3]{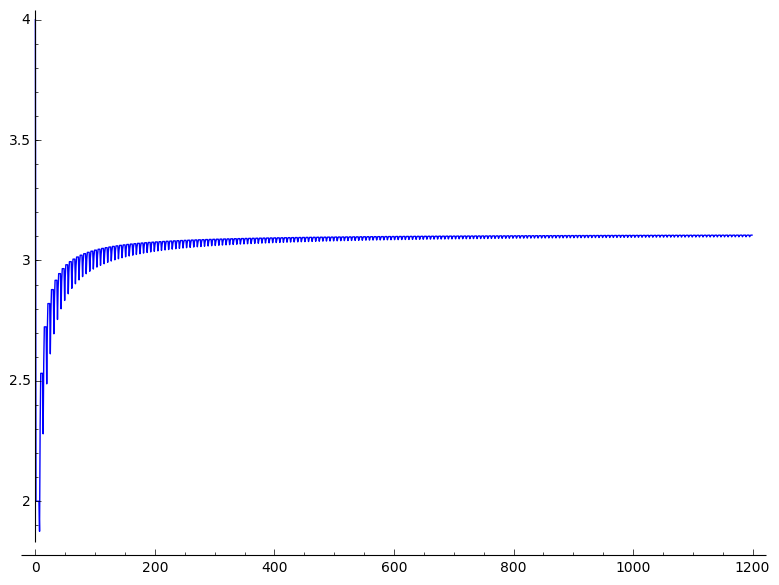}
}
\subfloat[$\frac{-2V(a(n))}{l(a(n))}$ for $\underline{a(n)}=101^20^2...1^n$]{
 \includegraphics[scale=0.3]{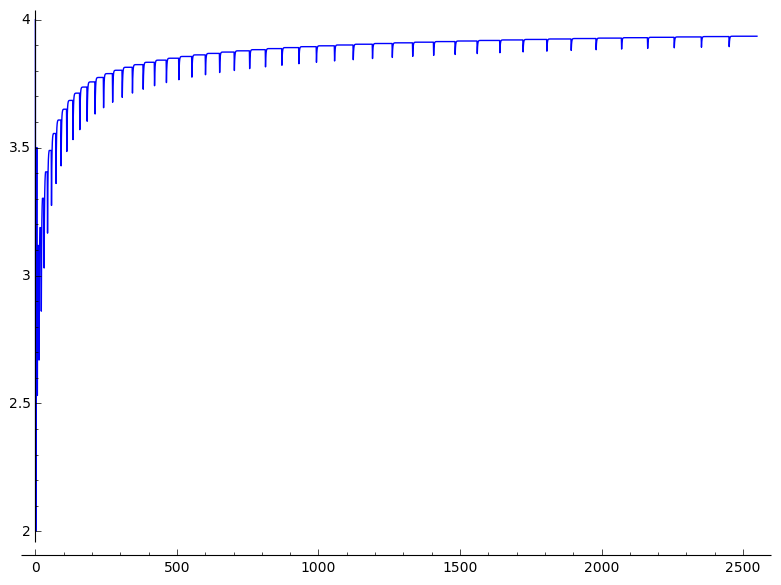}
 }
 \caption{Asymptotics of the ratio $\frac{-2V(a(n))}{l(a(n))}$ for different sequences $(a(n))_{n\in\mathbb{N}}$}
\label{f.ratio}
\end{figure}

\renewcommand{\abstractname}{Acknowledgements}
\begin{abstract}
We wish to thank the referee whose very helpful remarks played an important role in the improvement and clarification of this article. We also wish to thank Nicolas Bédaride and Pascal Hubert whose advice and careful readings of this article helped greatly.
\end{abstract}

\newpage

\end{document}